\documentclass[12pt,reqno]{amsart}
\usepackage[top=2cm,bottom=2cm,right=2.5cm,left=2.5cm]{geometry}
\usepackage{amssymb}
\usepackage{amsmath, amsthm, amscd, amsfonts, amssymb, graphicx, color}
\usepackage[bookmarksnumbered, colorlinks, plainpages]{hyperref}

\usepackage{hyperref}

\newtheorem{theorem}{Theorem}[section]
\newtheorem{lemma}[theorem]{Lemma}
\newtheorem{proposition}[theorem]{Proposition}

\theoremstyle{definition}
\newtheorem{definition}[theorem]{Definition}

\theoremstyle{remark}

\numberwithin{equation}{section}

\begin{document}
	
	\setcounter{page}{1}
	
	\title[Representations of $g$-fusion frames in Hilbert $C^{\ast}$-Modules]{Representations of $g$-fusion frames in Hilbert $C^{\ast}$-Modules}

	\author[A. Karara, R. Eljazzar, H. Sfouli]{Abdelilah Karara, Roumaissae Eljazzar$^{*}$ and  Hassan Sfouli}
	
	\address{Department of Mathematics, University of Ibn Tofail, Kenitra, Morocco}
	\email{\textcolor[rgb]{0.00,0.00,0.84}{abdelilah.karara@uit.ac.ma; roumaissae.eljazzar@uit.ac.ma; hassan.sfouli@uit.ac.ma}}

	
	
	\date{}

	\begin{abstract} In this paper, we provide some generalization of the concept of fusion frames following that evaluate their representability via a linear operator in Hilbert $C*$-module. We assume that $\Upsilon  _\xi$  is  self-adjoint and $\Upsilon  _\xi(\frak{N} _\xi)= \frak{N} _\xi$
		for all $\xi \in \mathfrak{S}$, and show that if a
		$g-$fusion frame $\{(\frak{N} _\xi, \Upsilon  _\xi)\}_{\xi \in \mathfrak{S}}$ is represented via a linear operator $\mathcal{T}$ on
		$\hbox{span} \{\frak{N} _\xi\}_{ \xi \in \mathfrak{S}}$, then $\mathcal{T}$ is bounded. Moreover, if $\{(\frak{N} _\xi, \Upsilon  _\xi)\}_{\xi \in \mathfrak{S}}$ is a  tight $g-$fusion frame, then $\Upsilon_\xi $  is not represented via an invertible linear operator on  $\hbox{span}\{\frak{N} _\xi\}_{\xi \in \mathfrak{S}}$, We show that, under certain conditions, a linear operator may also be used to express the perturbation of representable fusion frames. Finally, we'll investigate the stability of this fusion frame type.
	\end{abstract}
	\maketitle
	
	\baselineskip=12.4pt
	
	\section{Introduction and preliminaries}

	The concept of frames has emerged as a recent and active area of research in mathematics, as well as in fields like signal processing and computer science. The foundational work on frames for Hilbert spaces was laid out in 1952 by Duffin and Schaefer, particularly for the analysis of nonharmonic Fourier series, as cited in \cite{Duf}. This field gained renewed momentum in 1986 through the significant contributions of Daubechies, Grossmann, and Meyer, as referenced in \cite{Dau}, leading to a broader recognition and application since then.
	
	In more recent developments, numerous mathematicians have expanded the theory of frames beyond Hilbert spaces to encompass Hilbert $C^*$-modules. Detailed explorations of frames in this expanded context can be found in references such as \cite{r4, Kho2, r1, Jin, Pas}. Specifically, A. Khosravi and B. Khosravi have been instrumental in introducing concepts like fusion frames and g-frame theory within Hilbert $C^*$-modules, as discussed in \cite{Kho2}.

	This paper aims to explore $g$-fusion frames, denoted as $\{(\mathfrak{N}_\xi, \Upsilon_\xi)\}_{\xi \in \mathbb{Z}}$, through the lens of a linear operator, which may not be bounded.

	In this paper, we consider $\mathcal{H}$ as a countably generated Hilbert $\mathcal{A}$-module, with $\{\mathcal{K}_\xi\}_{\xi \in \mathbb{Z}}$ representing a set of Hilbert $\mathcal{A}$-modules and $\{\mathfrak{N}_\xi\}_{\xi \in \mathbb{Z}}$ as a sequence of closed orthogonally complemented submodules of $\mathcal{H}$. For each $\xi \in \mathbb{Z}$, we denote by $\operatorname{End}_\mathcal{A}^*(\mathcal{H}, \mathcal{K}_\xi)$ the set of all adjointable $\mathcal{A}$-linear mappings from $\mathcal{H}$ to $\mathcal{K}_{\xi}$, with the notation $\operatorname{End}_\mathcal{A}^*(\mathcal{H}, \mathcal{H})$ being equivalent to $\operatorname{End}_\mathcal{A}^*(\mathcal{H})$. The orthogonal projection onto the closed submodule $\mathfrak{N}_\xi$, which is orthogonally complementary within $\mathcal{H}$, is represented as $\mathcal{P}_{\mathfrak{N}_\xi}$. Furthermore, the notations $\mathcal{N}(\mathcal{T})$ and $\mathcal{R}(\mathcal{T})$ are utilized to signify the kernel and range of the operator $\mathcal{T}$, respectively.

	Let $\Upsilon_0$ be an operator in $\operatorname{End}_\mathcal{A}^*(\mathcal{H})$ with $\Upsilon_0(\mathcal{H}) \subseteq \operatorname{span} \{\mathfrak{N}_\xi\}_{\xi \in \mathbb{Z}}$, and $\mathcal{T}$ as the right-shift operator on
	$$ \ell^2\big( \{\mathfrak{N}_\xi\}_{\xi \in \mathbb{Z}} \big) = \left\{ \{f_\xi\}_{\xi \in \mathbb{Z}} : f_\xi \in \mathfrak{N}_\xi, \|\sum_{\xi \in \mathbb{Z}} \langle f_\xi, f_\xi \rangle \| < \infty \right\}.$$

	That is, the operation $\mathcal{T}$ on the sequence $\{f_\xi\}_{\xi \in \mathbb{Z}}$ yields the sequence $\{f_{\xi +1}\}_{\xi \in \mathbb{Z}}$. Considering two sequences $f = \{f_\xi\}_{\xi \in \mathbb{Z}}$ and $g = \{g_\xi\}_{\xi \in \mathbb{Z}}$, their inner product is articulated as $\langle f, g\rangle = \sum_{\xi \in \mathbb{Z}} \langle f_\xi, g_\xi\rangle$, thereby affirming that $\ell^2\big( \{\mathfrak{N}_\xi\}_{\xi \in \mathbb{Z}} \big)$ functions as a Hilbert $\mathcal{A}$-module.

We now proceed to succinctly revisit the fundamental concepts and characteristics of $C^*$-algebras and Hilbert $\mathcal{A}$-modules."

\begin{definition}\cite{Con}.
	Given a Banach algebra $\mathcal{A}$, we define an involution to be a function that maps an element $u$ to $u^{\ast}$ within $\mathcal{A}$. This function is characterized by the following properties for all elements $u, v \in \mathcal{A}$ and any scalar $\alpha$:
	\begin{enumerate}
		\item  $(u^{\ast})^{\ast}=u$.
		\item  $(uv)^{\ast}=v^{\ast}u^{\ast}$.
		\item  $(\alpha u+v)^{\ast}=\bar{\alpha}u^{\ast}+v^{\ast}$.
	\end{enumerate}
\end{definition}

\begin{definition}\cite{Con}.
	A $C^{\ast}$-algebra, denoted as $\mathcal{A}$, is a specific kind of Banach algebra that encompasses an involution. This algebra is distinctively defined by the property: $$\|u^{\ast}u\| = \|u\|^{2}$$ for every element $u$ in $\mathcal{A}$.
\end{definition}

\begin{definition}\cite{Kap}.
	Consider $\mathcal{A}$ as a unital $C^{\ast}$-algebra and $\mathcal{H}$ as a corresponding left $\mathcal{A}$-module, with compatible linear structures. $\mathcal{H}$ qualifies as a pre-Hilbert $\mathcal{A}$-module if it is equipped with a sesquilinear, positive definite $\mathcal{A}$-valued inner product, denoted $\langle ., . \rangle: \mathcal{H} \times \mathcal{H} \rightarrow \mathcal{A}$, that complies with the module action. Specifically, the following conditions are met:
	\begin{itemize}
		\item [(i)] For any $u \in \mathcal{H}$, the property $\langle u,u\rangle \geq 0$ holds true, and $\langle u,u\rangle = 0$ if and only if $u = 0$.
		\item [(ii)] The inner product satisfies $\langle \eta u+v,w\rangle = \eta\langle u,w\rangle + \langle v,w\rangle$ for all $\eta \in \mathcal{A}$ and for every $u, v, w \in \mathcal{H}$.
		\item [(iii)] The relation $\langle u,v\rangle = \langle v,u\rangle^{\ast}$ is valid for all $u, v \in \mathcal{H}$.
	\end{itemize}
	
	Additionally, for each element $u$ in $\mathcal{H}$, we define its norm as $\|u\| = \|\langle u,u\rangle\|^{\frac{1}{2}}$. $\mathcal{H}$, when complete under this norm, is recognized as a Hilbert $\mathcal{A}$-module or a Hilbert $C^{\ast}$-module over $\mathcal{A}$. For every $\eta$ in the $C^{\ast}$-algebra $\mathcal{A}$, we have $|\eta| = (\eta^{\ast}\eta)^{\frac{1}{2}}$.
\end{definition}

	\begin{definition}
		A sequence $\{\frak{N} _\xi\}_{\xi \in \mathbb{Z}}$  is  called a  \textit{fusion frame}  if there exist constants
		$0 <A \leq B< \infty$ and $v_\xi  > 0$ for $k = 1, 2, \ldots$  so that for every $f \in \mathcal{H}$
		\begin{eqnarray*} 
			A \langle f,f\rangle \leq  \sum_{\xi \in \mathbb{Z}} v_\xi ^2 \langle \mathcal{P}_{\frak{N} _\xi} (f),\mathcal{P}_{\frak{N} _\xi} (f)\rangle  \leq B  \langle f,f\rangle , 
		\end{eqnarray*}
	\end{definition}
	
	Now, let  $\mathcal{H}$ and  $\mathcal{K}$  be Hilbert $\mathcal{A}$-module spaces and  $\{\frak{N} _\xi\}_{\xi   \in  \mathfrak{S}} $  be a sequence of closed orthogonally complemented submodules of $\mathcal{K}$, where  $\mathfrak{S}$  is a subset of  $\mathbb{Z}$.
	
	\begin{definition}
		A  sequence  $\Lambda_\xi \in\operatorname{End}_\mathcal{A}^*\left(\mathcal{H}, \mathcal{K}_\xi\right)$  is called a {\it generalized frame  for $\mathcal{H}$ with respect to} $\{\frak{N} _\xi\}_{ \xi \in  \mathfrak{S} }$. If there are  positive constants $A$ and $B$ such that for all $f \in \mathcal{H}$:
		\begin{eqnarray*}
			A \, \langle  f,f\rangle \leq  \sum_{\xi \in \mathfrak{S}} \langle \Upsilon _\xi  f,\Upsilon _\xi  f\rangle  \leq B \, \langle f,f\rangle ,
		\end{eqnarray*}
		
	\end{definition}
	
	It should be noted that fusion frames are a specific type of g-frame. Hence, setting 
	
	$$\mathcal{H}=\mathcal{K}  \,\,\,\,\,\,\,\,\,\,\,\,  \text{and} \,\,\,\,\,\,\,\,\,\,\, \Upsilon_\xi  = v_\xi \mathcal{P}_{\frak{N} _\xi}.$$
	
	So, any operator whose range is contained in $\frak{N}_\xi$ for $\xi \in \mathfrak{S}$, can naturally be substituted for the orthogonal projections on $\frak{N}_\xi$ to provide an extension of the idea of fusion frame.

	\begin{definition}
		Let $\Upsilon_\xi \in\operatorname{End}_\mathcal{A}^*\left(\mathcal{H}, \frak{N}_\xi\right)$ the sequence $\{(\frak{N} _\xi, \Upsilon_\xi )\}_{\xi  \in  \mathfrak{S}}$ 
		is called a 
		$g-$fusion frame with respect to $\{\frak{N} _\xi \} _{ \xi  \in  \mathfrak{S}} $  ( $g-$fusion frame) if 
	if two distinct constants $A$ and $B$ are present where $0 < A \leq B < \infty$ such that
		\begin{eqnarray}  \nonumber
			A  \langle f,f\rangle  \leq  \sum_{\xi  \in  \mathfrak{S}}  \langle \Upsilon_\xi   (f ), \Upsilon_\xi   (f )\rangle \leq B  \langle f,f\rangle , \quad  f \in \mathcal{H}.
		\end{eqnarray}
		If, $A=B$, the sequence $\{(\frak{N} _\xi, \Upsilon_\xi )\}_{\xi  \in  \mathfrak{S}}$ is called a \emph{tight $g-$fusion frame}. 
	\end{definition}

	The $\emph{synthesis operator}$, $\emph{synthesis operator}$
	$ \mathcal{U}: \ell^2\big( \{\frak{N} _\xi \} _{ \xi  \in  \mathfrak{S}} \big) \rightarrow  \mathcal{H}$ is defined as:
	$$\mathcal{U} \bigg( \{f_\xi \}_{\xi  \in  \mathfrak{S}} \bigg) = \sum_{\xi  \in  \mathfrak{S}} \Upsilon^{*}_{\xi }  (f_\xi ).$$
	The frame operator $S$, mapping $\mathcal{H}$ to itself, is specified by:  
	\begin{eqnarray} \nonumber
	Sf = \sum_{\xi  \in  \mathfrak{S} } {\Upsilon}_{\xi }  ^{*}   {\Upsilon}_{\xi }   f   \nonumber
	\end{eqnarray}
	A $g-$fusion frame $\{(\frak{N} _\xi, \Gamma_\xi )\}_{\xi  \in  \mathfrak{S}}$ is termed a $\emph{dual $g-$fusion frame}$ of $\{(\frak{N} _\xi, \Upsilon_\xi )\}_{\xi  \in  \mathfrak{S}}$ if it fulfills:
	\begin{eqnarray} \nonumber
	f= \sum_{\xi  \in  \mathfrak{S}} \Upsilon_\xi ^{*}  {\Gamma} _\xi   f,  \,\,\,\,\,\,\,\,\,\,\,\,\,\,\,\,\,\,\,\,\,\,  (f \in \mathcal{H}).
	\end{eqnarray}
	
	It is evident that the sequence $\{(\frak{N} _\xi, \Upsilon_\xi  S^{-1})\}_{\xi  \in  \mathfrak{S}}$ forms a dual $g-$fusion frame of $\{(\frak{N} _\xi, \Upsilon_\xi )\}_{\xi  \in  \mathfrak{S}}$, known as the $\textit{canonical dual $g-$fusion frame}$.
	Note that for any $f \in \mathcal{H}$, it is observed that:
	\begin{eqnarray} \nonumber
	f = S^{-1}  S  f  =S^{-1} \sum_{\xi  \in  \mathfrak{S} } {\Upsilon}_{\xi }  ^{*}   {\Upsilon}_{\xi }f=\sum_{\xi  \in  \mathfrak{S}}   {\Upsilon_\xi}^{\ast}  \Upsilon_\xi    S^{-1} f.   \nonumber
	\end{eqnarray}
	
	Let $\mathcal{T}$ be a linear map on $\hbox{span}\{\frak{N} _\xi\}_{\xi  \in  \mathfrak{S}}$ and let $\Upsilon_\xi $ be a sequence in $\operatorname{End}_\mathcal{A}^*\left(\mathcal{H}\right)$, such that the range of $\Upsilon_\xi $ is within $\frak{N} _\xi$ for each $\xi  \in  \mathfrak{S}$.
	
	\begin{definition}
		We say that  $\{\Upsilon_\xi  \}_{\xi  \in  \mathfrak{S}}$
		is represented via $\mathcal{T}$ if for every $ \xi  \in  \mathfrak{S}$,
		$$ \mathcal{T}  \Upsilon_\xi=\Upsilon_{\xi +1}. $$
		
	\end{definition}
	
\begin{definition}
	A $g-$fusion frame, represented as $\{(\frak{N} _\xi, \Upsilon_\xi )\}_{\xi  \in  \mathfrak{S}}$, is considered to be associated with a linear operator $\mathcal{T}$ when each $\Upsilon_\xi$ within the frame is expressed via $\mathcal{T}$.
\end{definition}

	\section{On the  Boundedness of the operator $\mathcal{T}$}
	In the following section,  let $\Upsilon_\xi $ be a sequence in $\operatorname{End}_\mathcal{A}^*\left(\mathcal{H}\right)$ such that for every $\xi  \in  \mathfrak{S}$,  $\mathcal{R}(\Upsilon_\xi) $ is contained in $\frak{N} _\xi$.

	\begin{theorem}
		Assume for each $\xi  \in  \mathfrak{S}$, $\Upsilon_\xi (\frak{N} _\xi)= \frak{N} _\xi$ and $\Upsilon_\xi$ in $\operatorname{End}_\mathcal{A}^*\left(\mathcal{H}\right)$ is self-adjoint. If $\{(\frak{N} _\xi, \Upsilon_\xi )\}_{\xi  \in  \mathfrak{S}}$ constitutes a $g-$fusion frame, depicted by a linear operator $\mathcal{T}$ acting on $\hbox{span} \{\frak{N} _\xi\}_{ \xi  \in  \mathfrak{S}}$, then $\mathcal{T}$ is a bounded operator and  $\mathcal{N}(\mathcal{U})$ remains constant under the application of the right-shift operator. Additionally, the following inequality holds:
		\begin{eqnarray} \nonumber
		1 \leq \|\mathcal{T}\| \leq \sqrt{\dfrac{B}{A}},
		\end{eqnarray}
		with $A$ and $B$ as the frame bounds for $\{\Upsilon_\xi  \}_{\xi  \in  \mathfrak{S}}$.
	\end{theorem}
	
	\begin{proof}
		If $\{g_\xi  \}_{\xi  \in  \mathfrak{S}}$ is a sequence in $\mathcal{N}(\mathcal{U})^{\perp}$
		such that $g_\xi =0$ for all a finite number of $\xi$, then 
		\begin{eqnarray} \nonumber
			\langle \mathcal{T} \mathcal{U} \bigg( \{g_\xi \}_{\xi  \in  \mathfrak{S}} \bigg),\mathcal{T} \mathcal{U} \bigg( \{g_\xi \}_{\xi  \in  \mathfrak{S}} \bigg)\rangle &\leq & B  \sum_{\xi  \in  \mathfrak{S} } \langle g_\xi ,g_\xi \rangle  \nonumber \\ &\leq & \frac{B}{A}   \langle \mathcal{U}  ( \{g_\xi \}_{\xi  \in  \mathfrak{S}} ),\mathcal{U}( \{g_\xi \}_{\xi  \in  \mathfrak{S}} ) \rangle.\nonumber 
		\end{eqnarray}
		This implies that
		\begin{eqnarray}
			\langle \mathcal{T} \mathcal{U}  (\{g_\xi \}_{\xi  \in  \mathfrak{S}} + \{h_\xi \}_{\xi  \in  \mathfrak{S}}),\mathcal{T} \mathcal{U}  (\{g_\xi \}_{\xi  \in  \mathfrak{S}} + \{h_\xi \}_{\xi  \in  \mathfrak{S}})\rangle  \leq \frac{B}{A}   \| \mathcal{U}  (\{g_\xi  \}_{\xi  \in  \mathfrak{S}} +\{h_\xi  \}_{\xi  \in  \mathfrak{S}}  ) \|^2. 
		\end{eqnarray}
		For all  sequence $\{h_\xi \}_{\xi  \in  \mathfrak{S}}$ in $\mathcal{N}(\mathcal{U})$
		From (2.1) and $\hbox{span}  \{\frak{N} _\xi\}_{\xi  \in  \mathfrak{S}}$  is a subspace of  $\mathcal{R}(\mathcal{U})$, We conclude that for all  $f \in \hbox{span}\{\frak{N} _\xi\}_{\xi  \in  \mathfrak{S}} $,
		\begin{eqnarray} \nonumber
			\langle \mathcal{T}  (f),\mathcal{T}  (f)\rangle  \leq \frac{B}{A}  \langle f,f\rangle ,
		\end{eqnarray}
		Therefore, $\mathcal{T}$ is bounded and $$\|\mathcal{T}\| \leq \sqrt{\dfrac{B}{A}}.$$ 
		Since  $\{(\frak{N} _\xi, \Upsilon_\xi )\}_{\xi  \in  \mathfrak{S}}$  is a 
		$g-$fusion frame  represented via $\mathcal{T}$, there exists $f_0 \in \mathcal{H}$ such that  $\sum_{\xi  \in  \mathfrak{S}}\langle\Upsilon_\xi  (f_0),\Upsilon_\xi  (f_0) \rangle \not =0$. On the other  hand,
		\begin{eqnarray} \nonumber
			\sum_{\xi  \in  \mathfrak{S}}    \bigg  \langle   \Upsilon_\xi   (f_0), \Upsilon_\xi   (f_0)  \bigg \rangle  &=&  \sum_{\xi  \in  \mathfrak{S}}    \bigg  \langle   \Upsilon_{\xi +1}  (f_0), \Upsilon_{\xi +1}  (f_0)  \bigg \rangle  \\ \nonumber &=&   \sum_{\xi  \in  \mathfrak{S}}    \bigg  \langle   \mathcal{T}\Upsilon_\xi   (f_0),\mathcal{T} \Upsilon_\xi   (f_0)  \bigg \rangle       \\ \nonumber &\leq & \|\mathcal{T}\|^2   \sum_{\xi  \in  \mathfrak{S}}    \bigg  \langle   \Upsilon_\xi   (f_0), \Upsilon_\xi   (f_0)  \bigg \rangle . \nonumber 
		\end{eqnarray}
		This shows that $\|\mathcal{T}\| \geq 1$.

		To finish the proof, let $\{f_\xi \}_{\xi  \in  \mathfrak{S}} \in\mathcal{N}(\mathcal{U})$. Then $\sum_{\xi  \in  \mathfrak{S}} \Upsilon_\xi  (f_\xi  ) =0$. Since $\mathcal{T}$ is bounded and
		$\mathcal{T} \Upsilon_\xi = \Upsilon_{\xi +1}$, we have
		$$
		\begin{aligned}
			\sum_{\xi  \in  \mathfrak{S}} \Upsilon_{\xi }(f_{\xi  -1}) &= \sum_{\xi  \in  \mathfrak{S}} \Upsilon_{\xi +1} (f_{\xi })\\  &= \sum_{\xi  \in  \mathfrak{S}} \mathcal{T} \Upsilon_\xi (f_{\xi })\\&=0.
		\end{aligned}
		$$
		Since $\Upsilon_\xi $  be  self-adjoint
		for all $\xi  \in  \mathfrak{S}$.
		We conclude that $$\mathcal{U}  \mathcal{T} (\{f_\xi \}_{\xi  \in  \mathfrak{S}})=0.$$   That  is, $\mathcal{T}(\{f_\xi \}_{\xi  \in  \mathfrak{S}}) \in\mathcal{N}(\mathcal{U})$.
	\end{proof}
	\begin{theorem}
		Let $\{(\frak{N} _\xi, \Upsilon_\xi )\}_{\xi  \in  \mathfrak{S}}$ be a  tight $g-$fusion frame such that  $\Upsilon_\xi $ is self-adjoint and $\Upsilon_\xi (\frak{N} _\xi)=\frak{N} _\xi$
		for all $\xi  \in  \mathfrak{S}$.
		Then  $\Upsilon_\xi $  is not represented via an invertible linear operator $\mathcal{T}$ on  $\hbox{span}\{\frak{N} _\xi\}_{\xi  \in  \mathfrak{S}}$.
	\end{theorem}

	\begin{proof}
		suppose by the absurd that
		$\{ (\frak{N} _\xi , \Upsilon_\xi  )\}_{\xi  \in  \mathfrak{S}}$ be a $g-$fusion frame  represented via an invertible linear operator $\mathcal{T}$ on $\hbox{span} \{\frak{N} _\xi\}_{\xi  \in  \mathfrak{S}}$.
		
		So
		$\mathcal{T}^\xi\Upsilon_0=\Upsilon_\xi $
		and so
		$\Upsilon_{-\xi}=\mathcal{T}^{-\xi}\Upsilon_0$ for all $\xi  \in  \mathfrak{S}$.
		Similarly, replacing $\mathcal{T}$  by $\mathcal{T}^{-1}$ in the proof of Theorem 2.1, we get
		\begin{eqnarray} 
			1 \leq \|\mathcal{T}^{m}\| \leq \sqrt{\dfrac{B}{A}},
		\end{eqnarray}
		where $A$ and $B$ are  frame bounds of $\{ (\frak{N} _\xi , \Upsilon_\xi  )\}_{\xi  \in  \mathfrak{S}}$ and $m\in\{-1, 1\}$.
		
		Since $\{ (\frak{N} _\xi , \Upsilon_\xi  )\}_{\xi  \in  \mathfrak{S}}$ is a tight $g-$fusion frame,
		then  $A=B$. So by (2.2)
		\begin{eqnarray} \nonumber
			\|\mathcal{T}\|=\|\mathcal{T}^{-1}\|=1.
		\end{eqnarray}
		Therefore  for every $f \in \mathcal{H}$,
		
		$$
		\begin{aligned}
			\|\langle f,f \rangle \|^{\frac{1}{2}}&= \|\langle \mathcal{T} ^{-1}  \mathcal{T} f,\mathcal{T} ^{-1}  \mathcal{T} f \rangle\|^{\frac{1}{2}} \\ &\leq \|\langle Tf,Tf\rangle\|^{\frac{1}{2}} \\ &\leq \|\langle\ f,f\rangle\|^{\frac{1}{2}} .
		\end{aligned}
		$$
		
		This shows that  $\mathcal{\mathcal{T}}$ is an isometry.
		Thus  for every $f \in \mathcal{H}$ and $\xi  \in  \mathfrak{S}$,
		\begin{eqnarray} \nonumber
			\|\langle \Upsilon_\xi   f,\Upsilon_\xi   f\rangle \|^{\frac{1}{2}}= \|\langle \mathcal{T}^\xi  \Upsilon_0  f,\mathcal{T}^\xi \Upsilon_0  f \rangle \|^{\frac{1}{2}} = \|\langle \Upsilon_0 f,\Upsilon_0 f \rangle \|^{\frac{1}{2}}.
		\end{eqnarray}
		Hence
		\begin{eqnarray} \nonumber
			\sum_{\xi \in {\mathfrak{S}}} \langle \Upsilon_0 f,\Upsilon_0 f \rangle =\sum_{\xi \in {\mathfrak{S}}}\langle \Upsilon_\xi  f,\Upsilon_\xi  f \rangle \leq B \langle f,f\rangle .
		\end{eqnarray}
		It follows that
		$\sum_{\xi \in {\mathfrak{S}}}\langle \Upsilon_0 f,\Upsilon_0 f \rangle$
		is a convergent series. So
		\begin{eqnarray} \nonumber
			\|\langle \Upsilon_\xi   f,\Upsilon_\xi   f\rangle \|^{\frac{1}{2}}= \|\langle \Upsilon_0 f,\Upsilon_0 f \rangle \|^{\frac{1}{2}}=0
		\end{eqnarray}
		for all $f\in \mathcal{H}$ and $\xi  \in  \mathfrak{S}$.
		Therefore  for every $f\in \mathcal{H}$ we have
		$$
		f=\sum_{\xi \in {\mathfrak{S}}}\Upsilon_\xi ^*\Gamma_\xi  f=0,$$
		where $\{(\frak{N} _\xi, \Gamma_\xi ) \}_{\xi  \in  \mathfrak{S}}$ is a dual
		$g-$fusion frame of $\{(\frak{N} _\xi, \Upsilon_\xi )\}_{\xi  \in  \mathfrak{S}}$, a contradiction.
	\end{proof}
	\section{Stability and Linear Independence  }
	\begin{lemma}
		Let $\{\Upsilon_\xi  \}_{\xi  \in  \mathfrak{S}} $  be a sequence in  $\operatorname{End}_\mathcal{A}^*(\mathcal{H})$ with $\mathcal{R}(\Upsilon_\xi) \subset\frak{N} _\xi$ for all $\xi  \in  \mathfrak{S}$. If 
		$\{\Upsilon_\xi  \}_{\xi  \in  \mathfrak{S}} $ is represented via a  linear operator $\mathcal{T}$ on
		$ \hbox{span} \{\frak{N} _\xi \}_{\xi  \in  \mathfrak{S}} $
		and   $ \hbox{span} \{\Upsilon_\xi  \}_{\xi  \in  \mathfrak{S}} $
		is   infinite dimensional, then 
		$ \{\Upsilon_\xi  \}_{\xi  \in  \mathfrak{S}} $
		is linearly independent and infinite.
	\end{lemma}

	\begin{proof}
		Assume that
		$\{\Upsilon_\xi  \}_{\xi  \in  \mathfrak{S}} $
		is linearly dependent. Hence there exist constants
		$\delta_p, \ldots, \delta_q$ such that
		\begin{eqnarray} \nonumber
			\sum_{\xi =p}^{q} \delta_\xi  \Upsilon_\xi  =0
		\end{eqnarray}
		and  $ \delta_{\xi _0} \not = 0$, for some $ p \leq \xi_0 \leq q$. Set
		\begin{eqnarray} \nonumber
			\alpha= \min \bigg\{ \xi : \,\,\,  p \leq \xi \leq q, \delta_\xi  \not =0 \bigg \} \,\,\,\, \hbox{and}  \,\,\,\,\,
			b= \max \bigg\{ \xi : \,\,\, p \leq \xi \leq q, \delta_\xi  \not =0 \bigg \}.
		\end{eqnarray}
		Then
		\begin{eqnarray} \nonumber
			\Upsilon_{b} = \sum_{\xi  =\alpha}^{b-1} d_\xi   \Upsilon_\xi   \,\,\,\,\,\,\,\,\, \,\,\,\,\,\,\,\,\,   \hbox{and}  \,\,\,\,\,\,\,\,\, \,\,\,\,\,\,\,\,\,
			\Upsilon_{\alpha} = \sum_{\xi  =\alpha+1}^{b} d_\xi ^{\prime} \Upsilon_\xi 
		\end{eqnarray}
		for certain constants
		$\alpha_ \xi $ and $\alpha_ \xi ^{\prime}$, $\xi= \alpha, \ldots, b$.
		Thus for every $i \in \mathbb{N}$, we have
		
		$$
		\begin{aligned}
			\mathcal{T}^i  \Upsilon_{b} &= \sum_{\xi =\alpha} ^{b -1} d_\xi   \mathcal{T}^i  \Upsilon_\xi 
			\\ &=\sum_{\xi=\alpha} ^{b -1} d_\xi   \Upsilon_{\xi +i }
			\\ &= \sum_{\xi =\alpha+i} ^{b+ i -1} d_{\xi -i} \Upsilon_{\xi  }.
		\end{aligned}
		$$
		
		Similarly,
		$$\mathcal{T}^{-i}  \Upsilon_{b} =  \sum_{\xi =\alpha-i} ^{b-1-i } d_{\xi -i}^{'}   \Upsilon_{\xi  }.$$
		Therefore, in addition to
		$\Upsilon_\xi = \mathcal{T}^\xi  \Upsilon_0$ show that
		$\hbox{span} \big \{ \Upsilon_\xi  \big \}_{ \alpha, \ldots, b}$ is invariant under $\mathcal{T}$ and $\mathcal{T}^{-1}$. Also, we have
		\begin{eqnarray} \nonumber
			\hbox{span} \big\{ \Upsilon_\xi   \big\}_{ \xi  \in  \mathfrak{S} }=\hbox{span} \big \{ \Upsilon_\xi  \big \}_{ \alpha, \ldots, b}.
		\end{eqnarray}
		Hence,
		$ \hbox{span} \big \{ \Upsilon_\xi  \big\}_{\xi  \in  \mathfrak{S}}$
		is  finite dimensional, a contradiction.
	\end{proof}
	As an obvious consequence of Lemma 3.1, we get the following results.
	\begin{proposition} 
		Let $\{ (\frak{N} _\xi , \Upsilon_\xi  )\}_{\xi  \in  \mathfrak{S}}$ be a $g-$fusion frame such that $\Upsilon_\xi $ is self-adjoint and $\Upsilon_\xi (\frak{N} _\xi)=\frak{N} _\xi$. Assume that  $\mathcal{T}$ is a linear operator on $\hbox{span}\{\frak{N} _\xi\}_{\xi \in{\mathfrak{S}}}$ such that it has an extension to a  bounded linear operator $\tilde{\mathcal{T}}: \mathcal{H}\rightarrow \mathcal{H}$ with $\tilde{\mathcal{T}}
		\Upsilon_\xi ^*=\Upsilon_{\xi +1}^*$. Let  there exist $n_0  \in  \mathfrak{S}$ such that $\Upsilon_{n_0}$ is surjective  and for every $\xi  \in  \mathfrak{S}$.
		
		If $\{\Upsilon_\xi  \}_{\xi  \in  \mathfrak{S}}$ is represented via  $\mathcal{T}$,
		then $\{\Upsilon_\xi  \}_{\xi  \in  \mathfrak{S}}$ is linearly dependent.
	\end{proposition}
	
	\begin{proof}
		Suppose that   $\{(\frak{N} _\xi, \Gamma_\xi  )\}_{\xi  \in  \mathfrak{S}}$  be the dual   $g-$fusion frame of 
		$\{ (\frak{N} _\xi , \Upsilon_\xi  )\}_{\xi  \in  \mathfrak{S}}$.
		
		Thus  every $f \in \mathcal{H}$,
		\begin{eqnarray} \nonumber
			f = \sum_{\xi  \in  \mathfrak{S}}  \Upsilon_{\xi }^{*}    \Gamma_\xi   (f).
		\end{eqnarray}
		So $\tilde{\mathcal{T}}(f)= \sum_{\xi  \in  \mathfrak{S}}  \Upsilon_{\xi +1}^{*}    \Gamma_\xi    (f)$. This implies that
		\begin{eqnarray*} \nonumber
			\mathcal{T}(\Upsilon_j (f) ) =  \sum_{\xi  \in  \mathfrak{S}}  \Upsilon_{\xi +1}^{*}    \Gamma_\xi     \Upsilon_j  (f).
		\end{eqnarray*}
		For every $f \in \mathcal{H}$ we have
		\begin{eqnarray} 
			\Upsilon_{j+1}  (f)= \sum_{\xi  \in  \mathfrak{S}}  \Upsilon_{\xi +1}^{*}    \Gamma_\xi     \Upsilon_j  (f).
		\end{eqnarray}
		Choose $\theta, \theta^{'}  \in  \mathfrak{S}$ such that  $\theta \not = \theta^{'}$ and $ n_{0}  \not \in  \{\theta -1, \theta,\theta^{'}-1, \theta^{'} \}$.
		Let $\{\tilde{\Upsilon} _\xi\}_{\xi  \in  \mathfrak{S}}$ denote the sequence consisting of the same elements as $\Upsilon_\xi $, but with $\Upsilon_\theta$ and $\Upsilon_{\theta^\prime}$ interchanged.
		Clearly, that
		\begin{eqnarray} \nonumber
			\widetilde{\Gamma} _\xi =
			\left\{
			\begin{array}{ll}
				\Gamma_\xi , \,\,\,\,\,\,\,\,\,\,\,\,\,\,    \xi \not = \theta, \theta^{'},  \\
				\Gamma_{\theta^{'}},  \,\,\,\,\,\,\,\,\,\,\,\,\,\,     \xi=\theta, \\
				\Gamma_{\theta},  \,\,\,\,\,\,\,\,\,\,\,\,\,\,     \xi=\theta^{'},
			\end{array}
			\right.
		\end{eqnarray}
		is the canonical dual $g-$fusion frame $\{(\frak{N} _\xi, \widetilde{\Upsilon} _\xi ) \}_{ \xi  \in  \mathfrak{S}}$.
		From (3.1) we obtain
	
		\begin{eqnarray} 
			\Upsilon_{j_{0}+1} &=&  \sum_{\xi  \in  \mathfrak{S}} \Upsilon_{\xi +1}  {\Upsilon_\xi   S^{-1}}   \Upsilon_{j_{0}}  \end{eqnarray}
		Hence
		\begin{eqnarray}\nonumber
			\Upsilon_{j_{0}+1}&=&  \Upsilon_{\theta}  \Upsilon_{\theta-1}  S^{-1}   \Upsilon_{j_{0}}
			+  \Upsilon_{\theta+1}   \Upsilon _{\theta}   S^{-1}    \Upsilon_{j_{0}} 
			+  \Upsilon_{\theta^{'}}   \Upsilon_{\theta^{'}-1}   S^{-1}    \Upsilon_{j_{0}}
			+\Upsilon_{\theta^{'}+1}  \Upsilon_{\theta^{'}}   S^{-1}   \Upsilon_{j_{0}} \\ \nonumber &+& \sum_{\xi  \not \in  \{\theta -1, \theta, \theta^{'}-1, \theta{'}\} }\Upsilon_{\xi +1}   {\Upsilon_\xi   S^{-1}}  \nonumber \Upsilon_{j_{0}}     
		\end{eqnarray} 
		Then we have
		
		\begin{eqnarray} 
			\Upsilon_{j_{0}+1} &=&  \widetilde{\Upsilon}_{j_{0}+1} \nonumber \\
			&=&  \sum_{\xi  \in  \mathfrak{S}}  \widetilde{\Upsilon}_{\xi +1}  \widetilde{\Gamma} _\xi    \widetilde{\Upsilon }_{j_{0}} . \nonumber 
		\end{eqnarray}
		So,
		\begin{eqnarray}\nonumber 
			\Upsilon_{j_{0}+1} &=&  \Upsilon_{\theta ^{'}}   \Upsilon_{\theta-1}  S^{-1}   \Upsilon_{j_{0}}
			+  \Upsilon_{\theta+1} \Upsilon_{\theta^{'}}  S^{-1}   \Upsilon_{j_{0}} 
			+ \Upsilon_{\theta}   \Upsilon_{\theta^{'}-1}  S^{-1}  \Upsilon_{j_{0}}
			+ \Upsilon_{\theta^{'}+1}   \Upsilon_{\theta}  S^{-1}  \Upsilon_{j_{0}} \\ \nonumber &+& \sum_{\xi  \not \in  \{\theta -1, \theta, \theta^{'}-1, \theta{'}\} }\Upsilon_{\xi +1}  \Upsilon_\xi    S^{-1}  \Upsilon_{j_{0}}. \nonumber
		\end{eqnarray}                                                              
		Applying   (4.2), we achieve
		\begin{eqnarray} \nonumber
			\Upsilon_{\theta}  \Upsilon_{\theta-1}  S^{-1}   \Upsilon_{j_{0}}
			&+&  \Upsilon_{\theta+1}   \Upsilon _{\theta}   S^{-1}    \Upsilon_{j_{0}}
			+  \Upsilon_{\theta^{'}}   \Upsilon_{\theta^{'}-1}   S^{-1}    \Upsilon_{j_{0}} 
			+ \Upsilon_{\theta^{'}+1}   \Upsilon_{\theta^{'}}   S^{-1}   \Upsilon_{j_{0}} \nonumber \\
			&=&   \Upsilon_{\theta ^{'}}   \Upsilon_{\theta-1}  S^{-1}   \Upsilon_{j_{0}}
			+  \Upsilon_{\theta+1}   \Upsilon_{\theta^{'}}  S^{-1}   \Upsilon_{j_{0}}
			+ \Upsilon_{\theta}   \Upsilon_{\theta^{'}-1}  S^{-1}   \Upsilon_{j_{0}} 
			+ \Upsilon_{\theta^{'}+1}   \Upsilon_{\theta}   S^{-1}  \Upsilon_{j_{0}}. \nonumber
		\end{eqnarray}
		
		Now, we assume that $\theta >\theta^{'}$ and $\theta^{'}= \theta+q$,  for some $q \in \mathbb{N}$.
		So
		
		\begin{eqnarray} \nonumber
			0 &=&\Upsilon_{\theta}   \Upsilon_{\theta-1}   S^{-1}   \Upsilon_{j_{0}}
			+  \Upsilon_{\theta+1}  \Upsilon_{\theta} S^{-1}  \Upsilon_{j_{0}}
			+ \Upsilon_{\theta+q}  \Upsilon_{\theta+(q-1)}  S^{-1}  \Upsilon_{j_{0}} + \Upsilon_{\theta+(q+1)}  \Upsilon_{\theta+q}  S^{-1}   \Upsilon_{j_{0}}
			\nonumber \\
			&-& \Upsilon_{\theta+q} \Upsilon_{\theta-1}   S^{-1}   \Upsilon_{j_{0}} -  \Upsilon_{\theta+1}  \Upsilon_{\theta+q}   S^{-1}   \Upsilon_{j_{0}}
			- \Upsilon_{\theta}    \Upsilon_{\theta+(q-1)}  S^{-1} \Upsilon_{j_{0}} - \Upsilon_{\theta+(q+1)}   \Upsilon_{\theta}   S^{-1}  \Upsilon_{j_{0}}.   \nonumber
		\end{eqnarray}
		
		Hence
		\begin{eqnarray} \nonumber
			\Upsilon_{2  \theta-1}   S^{-1}   \Upsilon_{j_{0}}
			&+&  \Upsilon_{2   \theta+1}  S^{-1}  \Upsilon_{j_{0}}
			+ \Upsilon_{2 (\theta + q) -1}   S^{-1}  \Upsilon_{j_{0}}  \nonumber \\
			&+& \Upsilon_{2 ( \theta+q)+1}  S^{-1}   \Upsilon_{j_{0}}
			-2    \Upsilon_{2  \theta+(q-1)}  S^{-1}  \Upsilon_{j_{0}} \nonumber \\
			&-& 2   \Upsilon_{2  \theta+(q+1)}  S^{-1}   \Upsilon_{j_{0}} = 0.   \nonumber
		\end{eqnarray}
		Since  $\Upsilon_{j_0}$ and $S^{-1}$  are surjective, we have
		\begin{eqnarray} \nonumber
			\Upsilon_{2  \theta-1}  &+& \Upsilon_{2   \theta+1} + \Upsilon_{2  (\theta+q)-1} \nonumber \\
			&+&   \Upsilon_{2 ( \theta+q)+1}
			-  2  \Upsilon_{2 \theta+(q-1)}  \nonumber \\
			&- & 2  \Upsilon_{2  \theta+(q+1)} =0,   \nonumber
		\end{eqnarray}
		so $\Upsilon_\xi $ is linearly dependent. 
	\end{proof}

\begin{theorem}
	Consider the sequence $\{\widehat{\Upsilon} _\xi \}_{\xi  \in  \mathfrak{S}}$ within $\operatorname{End}_\mathcal{A}^*(\mathcal{H})$, where for each $\xi  \in  \mathfrak{S}$, $\widehat{\Upsilon} _\xi(\mathcal{H}) \subseteq \frak{N} _\xi$. Let $\{ (\frak{N} _\xi , \Upsilon_\xi  )\}_{\xi  \in  \mathfrak{S}}$ represent a $g-$fusion frame via a linear operator $\mathcal{T}$ acting on $\hbox{span}\{\frak{N} _\xi\}_{\xi \in {\mathfrak{S}}}$. Assume the existence of constants $\eta, \beta \in [0,1)$, fulfilling the inequality
	\begin{eqnarray}
	\bigg \| \sum_{\xi =1}^ n  \alpha_ \xi\big(\Upsilon_\xi  - \widehat{\Upsilon} _\xi  \big) (f) \bigg \| \leq \eta
	\bigg \| \sum_{\xi =1}^ n  \alpha_ \xi  \Upsilon_\xi   (f) \bigg \| + \beta
	\bigg \| \sum_{\xi =1}^ n  \alpha_ \xi \widehat{\Upsilon} _\xi (f) \bigg \|
	\end{eqnarray}
	for any $f \in \mathcal{H}$ and all finite complex sequences $\{\alpha_ \xi \}_{\xi =1}^n$. Under these conditions, the subsequent statements are valid:
	\begin{itemize}
		\item[$\emph{(i)}$]
		\begin{eqnarray} \nonumber
		\bigg(\frac{1 - \eta}{1 + \beta}\sqrt{A} \bigg)^2     \langle f,f\rangle \leq \sum_{ \xi  \in  \mathfrak{S}}  \langle\widehat{\Upsilon}_\xi (f),\Upsilon_\xi (f)  \rangle \leq  \bigg(\frac{1 + \eta}{1 - \beta}\sqrt{B} \bigg)^2     \langle f,f\rangle ,
		\end{eqnarray} where $A$ and $B$ are frame bounds of $\{ (\frak{N} _\xi , \Upsilon_\xi  )\}_{\xi  \in  \mathfrak{S}}$. 
		\item[$\emph{(ii)}$]
		If the span of $\Upsilon_\xi$ is infinite-dimensional and $\{\widehat{\Upsilon} _\xi \}_{\xi  \in  \mathfrak{S}}$ forms an infinite set, then $\{\widehat{\Upsilon} _\xi \}_{\xi  \in  \mathfrak{S}}$ exhibits linear independence.
	\end{itemize}
\end{theorem}

\begin{proof}
	Select $l  \in  \mathfrak{S}$ and define $\alpha_ \xi  = \varepsilon_{l\xi}$ in (3.3), where $$\varepsilon _{l\xi}={\begin{cases}0&\text{if } l\neq \xi,\\1&\text{if } l=\xi.\end{cases}}$$ Consequently, for any $f \in \mathcal{H}$
	\begin{eqnarray} \nonumber
	\| \Upsilon_l (f) - \widehat{\Upsilon}_l (f) \| \leq \eta  \| \Upsilon_l (f)  \|  + \beta  \| \widehat{\Upsilon}_l (f) \|.
	\end{eqnarray}
	Then
	\begin{eqnarray} \nonumber
	(1 - \beta)   \| \widehat{\Upsilon}_l (f) \| \leq (1 + \eta)  \| \Upsilon_l (f)  \|.
	\end{eqnarray}
	Thus
	\begin{eqnarray} \nonumber
	\langle\widehat{\Upsilon}_l (f),\Upsilon_l (f)  \rangle \leq  \bigg(\frac{1 + \eta}{1 - \beta} \bigg)^2   \langle\Upsilon_l (f),\Upsilon_l (f)  \rangle.
	\end{eqnarray}
	This implies
	\begin{eqnarray} \nonumber
	\sum_{ i  \in  \mathfrak{S}}  \langle\widehat{\Upsilon}_l (f),\Upsilon_l (f)  \rangle \leq  \bigg(\frac{1 + \eta}{1 - \beta} \bigg)^2   \sum_{ i  \in  \mathfrak{S}}   \langle\Upsilon_l (f),\Upsilon_l (f)  \rangle \leq  \bigg(\frac{1 + \eta}{1 - \beta} \bigg)^2   B  \langle f,f\rangle,
	\end{eqnarray}
	where $B$ is the upper frame bound of $\{ (\frak{N} _\xi , \Upsilon_\xi  )\}_{\xi  \in  \mathfrak{S}}$.
	
	Similarly, for all $i  \in  \mathfrak{S}$ and $f \in \mathcal{H}$, $$ \langle\widehat{\Upsilon}_l (f),\Upsilon_l (f)  \rangle\geq  \bigg(\frac{1 - \eta}{1 + \beta} \bigg)^2  \langle \Upsilon_l f,\Upsilon_l f\rangle.$$
	Hence
	$$ \sum_{ i  \in  \mathfrak{S}}  \langle\widehat{\Upsilon}_l (f),\Upsilon_l (f)  \rangle\geq  \bigg(\frac{1 - \eta}{1 + \beta} \bigg)^2  \sum_{ i  \in  \mathfrak{S}} \langle \Upsilon_l f,\Upsilon_l f\rangle.$$
	In conclusion,
	\begin{eqnarray} \nonumber
	\sum_{ i  \in  \mathfrak{S}}  \langle\widehat{\Upsilon}_l (f),\Upsilon_l (f)  \rangle \geq  \bigg(\frac{1 - \eta}{1 + \beta} \bigg)^2  A   \langle f,f\rangle.
	\end{eqnarray}
	
	That is, \begin{eqnarray} \nonumber
	\bigg(\frac{1 - \eta}{1 + \beta}\sqrt{A} \bigg)^2     \langle f,f\rangle \leq \sum_{ \xi  \in  \mathfrak{S}}  \langle\widehat{\Upsilon}_\xi (f),\Upsilon_\xi (f)  \rangle \leq  \bigg(\frac{1 + \eta}{1 - \beta}\sqrt{B} \bigg)^2     \langle f,f\rangle,
	\end{eqnarray}
	
	Let $\{\alpha_ \xi  \}_{\xi =1}^n$ be a finite sequence in $\mathbb{C}$ with  $\sum_{\xi =1}^n \alpha_ \xi   \widehat{\Upsilon} _\xi =0$. Given $\eta \in [0,1)$,  by  (3.3)
	\begin{eqnarray} \nonumber
	\sum_{\xi =1}^n \alpha_ \xi  \Upsilon_\xi  =0.
	\end{eqnarray}
	As per Proposition 3.1, $\alpha_ \xi  =0$ for all $k=1, \ldots, n$.
	Thus, $\{\widehat{\Upsilon} _\xi \}_{\xi  \in  \mathfrak{S}}$ is linearly independent.
\end{proof}

	\medskip

	\section*{Declarations}
	
	\medskip
	
	\noindent \textbf{Availability of Data and Materials}\newline
	\noindent Not applicable.
	
	\medskip
	
	\noindent \textbf{Ethics Approval and Consent to Participate}\newline
	\noindent It is important to note that this article does not involve any studies with animals or human participants.
	
	\medskip
	
	\noindent \textbf{Competing Interests}\newline
	\noindent The authors have no conflicts of interest to declare.
	
	\medskip
	
	\noindent \textbf{Funding}\newline
	\noindent There are no financial sources to declare for this paper.
	
	\medskip
	
	\noindent \textbf{Authors' Contributions}\newline
	\noindent All authors have equally contributed to the conception and design of the study, drafting of the manuscript, sequence alignment, and have read and approved the final manuscript.

\end{document}